\documentclass[11pt,reqno]{amsart}

\title{Self-Reachable Configuration Polytopes for Trees}

\author{Benjamin Lyons, McCabe Olsen}
\address{Department of Mathematics\\
  Rose-Hulman Institute of Technology\\
  Terre Haute, IN 47803--3920}
\date{\today}
\thanks{The first author was a participant at the Rose-Hulman REU funded by NSF DMS \#1852132 and advised by the second author.
The second author was partially supported by NSF DMS \# 2418532.}


\usepackage{amsmath,amssymb, amscd, amsmath, bm, mathtools}
\usepackage{microtype}
\usepackage{mathrsfs}
\usepackage{graphics}
\usepackage{graphicx}
\usepackage{tikz}
\usepackage{pdfpages}
\usepackage{pgfplots}
\usepackage{pgfplotstable}
\usepackage{tkz-fct}
\usepackage{latexsym}
\usepackage{graphicx}
\usepackage{multirow}
\usepackage{hyperref}
\usepackage{subfig}
\usepackage{enumitem}
\usepackage[bottom]{footmisc}
\usepackage[margin=1.25in]{geometry}
\usetikzlibrary{pgfplots.polar}
\pgfplotsset{compat=newest}
\usepgfplotslibrary{fillbetween}
\usetikzlibrary{patterns}
\input epsf

\numberwithin{equation}{section}
\numberwithin{figure}{section}
\newtheorem{theorem}{Theorem}[section]
\newtheorem{lemma}[theorem]{Lemma}
\newtheorem{corollary}[theorem]{Corollary}
\newtheorem{proposition}[theorem]{Proposition}
\theoremstyle{definition}
\newtheorem{definition}[theorem]{Definition}
\newtheorem{example}[theorem]{Example}
\theoremstyle{remark}
\newtheorem{remark}[theorem]{Remark}

\newcommand{\conv}{\text{conv}}

\def\XXint#1#2#3{{\setbox0=\hbox{$#1{#2#3}{\int}$ }
\vcenter{\hbox{$#2#3$ }}\kern-.6\wd0}}

\newcommand{\Def}[1]{\textbf{#1}}
\newcommand{\Z}{\mathbb{Z}}
\newcommand{\R}{\mathbb{R}}

\newcommand{\CP}[2]{\mathcal{CP}_{#1}^{(#2)}}

\begin{document}
\begin{abstract}
We study lattice polytopes which arise as the convex hull of chip vectors for \emph{self-reachable} chip configurations on a tree $T$.
We show that these polytopes always have the integer decomposition property and characterize the vertex sets of these polytopes. 
Additionally, in the case of self-reachable configurations with the smallest possible number of chips, we show that these polytopes are unimodularly equivalent to a unit cube.  
\end{abstract}

\maketitle

\section{Introduction}
Let $G$ be a graph.
The chip firing game on $G$ is a rather simple way to explain a discrete dynamical system on the graph: To each vertex, associate a nonnegative integer (often viewed as a commodity such as poker chips). Then choosing a vertex which has at least as many chips as its degree, we ``fire" the vertex. That is, we transfer one chip to each of its neighbors.
We then continue this process as long as we want or are able to. 
The notion of chip firing developed both as a combinatorial game (see, e.g. \cite{ALS89,BLS91,Spencer}) and via the abelian sandpile model, which was first studied by Bak, Tang, and Wisenfeld (see \cite{BakTangWisenfeld}).
The mathematics of chip firing is quite rich encompassing combinatorics, statistical physics, and algebraic geometry via Riemann-Roch theory (see \cite{Klivans} for a broad overview of chip firing).

In this paper, we study a family of lattice polytopes which arise as the convex hull certain chip configurations, called \emph{self-reachable configurations}, which was first defined and explored in \cite{First Paper}.
For a tree $T$, we denote the convex hull of self-reachable configurations with $\ell$ chips by $\CP{\ell}{T}$.
In Section~\ref{sec:background}, we discuss all necessary background for graph theory, chip firing, self-reachable configurations, and convex polytopes.
Section~\ref{sec:cubes} explores the polytopes $\CP{n-1}{T}$ which arise from the self-reachable configurations with the fewest number of chips possible.
We show that these polytopes are always unimodularly equivalent to unit cubes. 
In Section~\ref{sec:IDP}, we show that for any tree $\CP{\ell}{T}$ has the integer decomposition property. Finally, Section~\ref{sec:vertices} characterizes the vertices of the lattice polytopes $\CP{\ell}{T}$.

\section{Background}\label{sec:background}
We begin by setting notation which will be used throughout.

\begin{itemize}
	\item For $a, b \in \mathbb{Z}$, we define the sequences $\mathbb{N}_{a} = \{a, a + 1, a + 2, \ldots\}$, and $\mathbb{N}_{[a, b]} = \{a, a + 1, a + 2, \ldots, b\}$.
        \item Given two finite sequences $\Phi$ and $\Psi$, we let $\{\Phi, \Psi\}$ be the sequence obtained by appending the sequence $\Psi$ to the end of the sequence $\Phi$.
        \item For a vector $\bm{x} \in \mathbb{R}^n$, we clarify the dimension of $\bm{x}$ by writing $\bm{x}^{(n)}$.
        \item Given vectors $\bm{x} = (x_1, x_2, \ldots, x_n) \in \mathbb{R}^n$ and $\bm{y} = (y_1, y_2, \ldots, y_m) \in \mathbb{R}^m$, we let \[(\bm{x}, \bm{y}) = (x_1, x_2, \ldots, x_n, y_1, y_2, \ldots, y_m) \in \mathbb{R}^{n + m}.\]
        \item Unless otherwise specified, we will label the vertices of an $n$-vertex graph as $v_1, v_2, \ldots, v_n$. Throughout this paper, we will only work with graphs that have at least one vertex.
        \item We will let $\deg^{(G)}(v_a)$ refer to the degree of the vertex $v_a$ on the graph $G$.
        When the graph is clear from context, we will denote this as simply $\deg(v_a)$. 
        \item We will denote the vertex set of a graph $G$ by $V(G)$. 
\end{itemize}



We will now review some basic definitions from graph theory and chip firing. 
For additional background and detail beyond what is provided here, we refer the reader to the excellent book of Klivans \cite{Klivans}.
Let a simple graph $G$ on $n$ vertices.
Recall that the \Def{Laplacian matrix} of $G$, denoted $\Delta(G)$, is the $n \times n$ matrix given by

     \begin{align*}
         \Delta_{ij}(G) = 
         \begin{cases}
             -1 & i \neq j \text{ and } (v_i, v_j) \in E \\
             \deg(v_i) & i = j \\
             0 & \text{otherwise.}
         \end{cases}
     \end{align*}
A \Def{chip configuration} on $G$ is a vector $\bm{c}^{n}\in\Z_{\geq 0}^n$, where we interpret $c_i$ as the number of chips on vertex $v_i$.
If $G$ has $n$ vertices with chip configuration $\bm{c}^{(n)}$, the vector $\bm{d}^{(n)}$ that results from \Def{firing} the $i$th vertex $v_i$ is
	\begin{align*}
        \bm{d}^{(n)} = \bm{c}^{(n)} - \Delta(G)\bm{e}_i^{(n)}.
    \end{align*}
This has the effect of moving a chip from $v_i$ to each vertex neighboring $v_i$. 
We say that it is \Def{legal} to fire $v_i$ on $G$ starting from $\bm{c}^{(n)}$ if $\bm{d}^{(n)}$ is a chip configuration, which is to say the $i$th component of $\bm{d}^{(n)}$ remains nonnegative. 
Equivalently, it is legal to fire $v_i$ provided $\bm{e}_i^{(n)} \cdot \bm{c}^{(n)} \ge \deg(v_i)$.
A \Def{firing sequence} is a sequence of integers which represents a sequence of legal chip firings.
Throughout this paper, firing sequences will be assumed to be finite. 
Given a starting configuration $\bm{c}^{(n)}$ and a firing sequence $\Phi$, the resulting configuration is
	\begin{align}\label{eq2.1}
        \bm{F}_\Phi^{(G)}\left(\bm{c}^{(n)}\right) = \bm{c}^{(n)} - \Delta(G)\sum_{j \in \Phi}\bm{e}_j^{(n)}.
    \end{align}
We say that a firing sequence $\Phi = \{\phi_1, \phi_2, \ldots, \phi_m\}$ is \Def{legal} on $G$ starting from a chip configuration $\bm{c}^{(n)}$ if it consists entirely of legal vertex firings.
That is, it is legal to fire $v_{\phi_1}$ on $G$ starting from $\bm{c}^{(n)}$, it is legal to fire $v_{\phi_2}$ on $G$ starting from $\bm{F}_{\{\phi_1\}}^{(G)}\left(\bm{c}^{(n)}\right)$, it is legal to fire $v_{\phi_3}$ on $G$ starting from $\bm{F}_{\{\phi_1, \phi_2\}}^{(G)}\left(\bm{c}^{(n)}\right)$, and so on. We let $R^{(G)}\left(\bm{c}^{(n)}\right)$ denote the set of chip configurations on $G$ that can be reached via nonempty legal firing sequences on $G$ starting from $\bm{c}^{(n)}$.

\subsection{Self-Reachable Chip Configurations} 
 In this subsection, we will review the notion of a self-reachable configuration. 
For additional background and for proofs of any results stated within, please consult \cite{First Paper}.
We begin with the definition.
 
 \begin{definition}\label{selfreachableconfig}
    Let $G$ be a connected simple $n$-vertex graph. 
    The chip configuration $\bm{s}^{(n)}$ is called \Def{self-reachable} on $G$ if $\bm{s}^{(n)} \in R^{(G)}\left(\bm{s}^{(n)}\right)$. 
    Wet $S_\ell^{(G)}$ denote set of all self-reachable configurations on $G$ with exactly $\ell$ chips.
\end{definition}

The following result applies to any connected, simple graph equipped with a self-reachable configuration. 

\begin{proposition}[\cite{First Paper}]\label{Add Chips to SRC}
    Let $G$ be a connected simple $n$-vertex graph, and let $\bm{s}^{(n)}$ be a self-reachable configuration on $G$. Let $\bm{c}^{(n)}$ be a chip configuration on $G$ such that $\bm{c}^{(n)}$ has at least as many chips on each vertex of $G$ as $\bm{s}^{(n)}$. Then $\bm{c}^{(n)}$ is self-reachable on $G$.
\end{proposition}

The following collection of results specifically pertain to self-reachable configurations on trees. 
These results in particular will be useful in our eventual discussion of polytopes.  

\begin{theorem}[\cite{First Paper}]\label{srcsubtreeresult}
    Let $T$ be an $n$-vertex tree. Then the configuration $\bm{c}^{(n)}$ is self-reachable on $T$ if and only if $\bm{c}^{(n)}$ has at least $m - 1$ chips on every $m$-vertex subtree of $T$.
\end{theorem}

\begin{corollary}[\cite{First Paper}]\label{subtreeconfigsaresrcs}
    Let $T$ be an $n$-vertex tree, and let $\bm{s}^{(n)}$ be a self-reachable configuration on $T$. Then the chip configuration formed by $\bm{s}^{(n)}$ on any subtree $T^*$ of $T$ must be self-reachable on $T^*$.
\end{corollary}

\begin{corollary}[\cite{First Paper}]\label{AddChipToLeafOrNeighbor}
    Let $T$ be an $n$-vertex tree with $n\geq 2$ where $v_n$ is a leaf connected to $v_{n - 1}$. Then the following are equivalent:

    \begin{enumerate}
        \item $\bm{s}^{(n - 1)}$ is a self-reachable configuration on $T \setminus v_{n}$.

        \item $\left(\bm{s}^{(n - 1)} + \bm{e}_{n - 1}^{(n - 1)}, 0\right)$ is a self-reachable configuration on $T$.

        \item $\left(\bm{s}^{(n - 1)}, 1\right)$ is a self-reachable configuration on $T$.
    \end{enumerate}
\end{corollary}

We close this subsection by citing an enumerative result for self-reachable configurations on trees. 

\begin{theorem}[\cite{First Paper}]\label{srccountingtheorem}
    Let $T$ be an $n$-vertex tree and let $\displaystyle C_{\ell,n}\coloneq \left|S_\ell^{(T)}\right|$.    
    Then $C_{\ell,n}$ satisfies  recurrence relation for $\ell \in \mathbb{N}_2$ and $n \in \mathbb{N}_2$:
    \begin{align}\label{eq4.1}
        C_{\ell, n} = C_{\ell - 1, n} + 2C_{\ell - 1, n - 1} - C_{\ell - 2, n - 1}.
    \end{align}
    with $C_{\ell, 1} = 1$ for all $\ell \in \mathbb{N}_0$, $C_{0, n} = 0$ for all $n \in \mathbb{N}_2$, $C_{1, 2} = 2$, and $C_{1, n} = 0$ for all $n \in \mathbb{N}_3$. 
\end{theorem}

\subsection{Convex Polytopes}

We conclude with some brief background on polyhedral geometry. A \Def{polytope} $P\subset \R^n$ is convex hull of finitely many $\bm{x}_1^{(n)}, \bm{x}_2^{(n)}, \ldots, \bm{x}_k^{(n)}\in \R^n$. That is,
	\[
	P= \conv\left(\left\{\bm{x}_1^{(n)}, \bm{x}_2^{(n)}, \ldots, \bm{x}_k^{(n)}\right\}\right)\coloneq \left\{\bm{y}=\sum_{i=1}^k \gamma_i \bm{x}_i^{(n)} \ : \ 0\leq \gamma_i\leq 1 \mbox{ and } \sum_{i=1}^k \gamma_i=1\right\}
	\] 
 The inclusion-minimal set $V\subset \R^n$ such that $P = \conv(V)$ is called the \Def{vertex set} of $P$. 
A polytope $P$ is called \Def{lattice} (or \Def{integral}) if $V\subset \Z^n$. 
 For $t \in \mathbb{N}_1$, the \Def{$t$th dilation} of $P$ is the set $tP \coloneq \left\{t\bm{x}^{(n)} : \bm{x}^{(n)} \in P\right\}$. 
 A lattice polytope $P$ has the \Def{integer decomposition property}  if for all $t \in \mathbb{N}_1$ and $\bm{w}^{(n)} \in tP \cap \mathbb{Z}^n$, there exist $\bm{x}_1^{(n)}, \bm{x}_2^{(n)}, \ldots, \bm{x}_t^{(n)} \in P \cap \mathbb{Z}^{n}$ with 
\begin{equation*}
    \bm{w}^{(n)} = \sum_{j = 1}^{t}\bm{x}_j^{(n)}.
\end{equation*}
We say that two polytopes $P$ and $Q$ are \Def{unimodularly equivalent} if there exist $U \in \text{SL}_n(\mathbb{Z})$ and $\bm{b}^{(n)} \in \mathbb{Z}^n$ such that $Q = UP + \bm{b}^{(n)}$. For further background and details on the basics of lattice polytopes, we encourage the reader to see \cite{BeckRobins}.

In this paper, we will consider the polytopes which are generated as the convex hulls of self-reachable configurations on trees. For a tree $T$ and $\ell\in\mathbb{N}_1$, define
	\[
	\CP{\ell}{T} \coloneq \conv\left(S_\ell^{(T)} \right).
	\] 
We call these polytopes \Def{self-reachable configuration polytopes}.
One should note if $T$ has $n$ vertices that $\CP{\ell}{T}$ will have dimension at most $n-1$, as these polytopes by construction lie on the hyperplane $\sum_{i=1}^nx_i=\ell$.

\section{Minimally Self-Reachable Configurations and Unit Cubes}\label{sec:cubes}

In this section, we first consider a simple case of $\CP{\ell}{T}$. Specifically, we will consider these polytopes for {minimally self-reachable configurations}, which are defined as follows. 

\begin{definition}\label{minimalsrc}
    Let $T$ be a  $n$-vertex tree, and let $\bm{\mu}^{(n)}$ be a self-reachable configuration on $T$. We say that $\bm{\mu}^{(n)}$ is a \Def{minimally self-reachable configuration} on $T$ if $\bm{\mu}^{(n)}$ is self-reachable on $G$ and has exactly $n - 1$ chips on $T$.
\end{definition}

The term {minimally self-reachable} is chosen because \emph{any} configuration on an $n$-vertex tree with $\ell\leq n-2$ chips will never be self-reachable by Theorem \ref{srcsubtreeresult}.
Additionally, one should note that a chip configuration with $n-1$ chips need not necessarily be a self-reachable configuration. 
An easy example of this behavior would be to consider any tree $T$ with $n\geq 3$ vertices where all chips are placed on a leaf. 
That being said, we can more precisely ascertain necessary conditions for minimally self-reachable configurations as in the following lemma.

\begin{lemma}\label{nomorechipsthandegree}
    Let $T$ be an $n$-vertex tree, and let $\bm{\mu}^{(n)}$ be a minimally self-reachable configuration on $T$. Then on any vertex of $T$, $\bm{\mu}^{(n)}$ cannot have any more chips than the degree of that vertex. Equivalently, for all $i \in \mathbb{N}_{[1,n]}$,
    \begin{equation}
        \bm{e}_i^{(n)} \cdot \bm{\mu}^{(n)} \le \deg(v_i).\notag
    \end{equation}
\end{lemma}

\begin{proof}
    Suppose towards contradiction that on some vertex $v_i$,
    \begin{equation*}
        \bm{e}_i^{(n)} \cdot \bm{\mu}^{(n)} > d,
    \end{equation*}
    where we let $d = \deg(v_i)$. The graph $T \setminus v_i$ will contain $d$ disjoint trees, which we will label $T_1, T_2, \ldots, T_d$. For each $j \in \mathbb{N}_{[1,n]}$, let $m_j$ be the number of vertices on $T_j$. Since $\bm{\mu}^{(n)}$ is self-reachable on $T$, $\bm{\mu}^{(n)}$ must contain at least $m_j - 1$ chips on $T_j$ for each $j \in \mathbb{N}_{[1,n]}$ by Theorem \ref{srcsubtreeresult}. Therefore, the number of chips in $\bm{\mu}^{(n)}$ has on $T$ satisfies the following  bound:
    \begin{equation}
        \bm{1}^{(n)} \cdot \bm{\mu}^{(n)} > d + \sum_{j = 1}^{d}(m_j - 1) = d + (n - 1) - d = n - 1.
    \end{equation}
    This contradicts the minimally self-reachable configuration assumption.
    
\end{proof}

Since minimally self-reachable configurations are much more restrictive than general self-reachable configurations, we can explicitly enumerate them.  

\begin{lemma}\label{countminimalsrcs}
    Let $T$ be an $n$-vertex tree. Then
    \begin{equation*}
        \left|S_{n - 1}^{(T)}\right| = 2^{n - 1}.
    \end{equation*}
\end{lemma}

\begin{proof}
    Adopting notation from Theorem \ref{srccountingtheorem}, we have that $\left|S_{n - 1}^{(T)}\right| = C_{n - 1, n}$. By Theorem \ref{srccountingtheorem},
    \begin{align*}
        C_{n - 1, n} = C_{n - 2, n} + 2C_{n - 2, n - 1} - C_{n - 3, n - 1}.
    \end{align*}
	By Theorem \ref{srcsubtreeresult}, $C_{n-2,n}=C_{n-3,n-1}=0$.
	This reduces the recurrence to $C_{n - 1, n} = 2C_{n - 2, n - 1}$ and the result follows by the initial conditions $C_{0, 1} = 1$ and $C_{1, 2} = 2$ given in Theorem \ref{srccountingtheorem}.

\end{proof}

In terms of the polytope $\CP{n-1}{T}$, we can provide the  following explicit description via an unimodular equivalence with a well-known polytope: the unit cube. 


\begin{theorem}\label{cubeunimodularequivalence}
    Let $T$ be an $n$-vertex tree. Then $\CP{n-1}{T} \cong [0, 1]^{n - 1} \times \{0\}$.
\end{theorem}


\begin{proof}
	Notice that when $n=1$, $\CP{n-1}{T}=\{0\}$. 
	For $n\geq 2$, without loss of generality, suppose that $v_n$ is a leaf connected to $v_{n-1}$.
	Note that it sufficient to show  
	\begin{equation}
    \CP{n-1}{T} \cong \CP{n-2}{T\setminus v_n} \times [0,1]
    \end{equation} 
	as the result will follow by induction.
	For simplicity of notation, denote $Q=\CP{n-2}{T\setminus v_n} \times [0,1]$.
    %
%
    By Lemma \ref{countminimalsrcs}, we have that $\left|S_{n - 2}^{(T \setminus v_n)}\right| = 2^{n - 2}$. Hence, we let
    \begin{align*}
        S_{n - 2}^{(T \setminus v_n)} = \left\{\bm{s}_1^{(n - 1)}, \bm{s}_2^{(n - 1)}, \ldots \bm{s}_{2^{n - 2}}^{(n - 1)}\right\}.
    \end{align*}
	%
	Note that an arbitrary element of $Q$ can be represented as
    \begin{align}\label{eq6.1}
        \bm{y}^{(n)} = \sum_{i = 1}^{2^{n - 2}}\kappa_i\left(\bm{s}_i^{(n - 1)}, 1\right) + \sum_{i = 1}^{2^{n - 2}}\lambda_i\left(\bm{s}_i^{(n - 1)}, 0\right),
    \end{align}
    where $0 \le \kappa_i, \lambda_i$ for all $i \in \mathbb{N}_{[1,2^{n-2}]}$ and
    \begin{align*}
        \sum_{i = 1}^{2^{n - 2}}\kappa_i + \sum_{i = 1}^{2^{n - 2}}\lambda_i = 1.
    \end{align*}
    We now let
    \begin{align*}
        A &= \Big\{\left(\bm{s}_1^{(n - 1)}, 1\right), \left(\bm{s}_2^{(n - 1)}, 1\right), \ldots, \left(\bm{s}_{2^{n - 2}}^{(n - 1)}, 1\right), \\
        &\left(\bm{s}_1^{(n - 1)} + \bm{e}_{n - 1}^{(n - 1)}, 0\right), \left(\bm{s}_2^{(n - 1)} + \bm{e}_{n - 1}^{(n - 1)}, 0\right), \ldots, \left(\bm{s}_{2^{n - 2}}^{(n - 1)} + \bm{e}_{n - 1}^{(n - 1)}, 0\right)\Big\}.
    \end{align*}
    By Corollary \ref{AddChipToLeafOrNeighbor}, we have that $S_{n - 1}^{(T)} \subseteq A$. Note that all of the chip configurations in $A$ are distinct because the $\bm{s}_i^{(n - 1)}$s are distinct. So $A$ contains $2^{n - 1}$ distinct self-reachable configurations on $T$, each with $n - 1$ chips. Therefore, Lemma \ref{countminimalsrcs} implies that $S_{n - 1}^{(T)} = A$.

    
    We now let $U = I_n + E_{n - 1, n}$, where $I_n$ denotes the $n \times n$ identity matrix and $E_{i, j}$ denotes the $n \times n$ matrix with a 1 in row $i$ column $j$ and 0s in all other entries, and we let $\bm{b}^{(n)} = -\bm{e}_{n - 1}^{(n)}$. We claim that for this choice of $U$ and $\bm{b}^{(n)}$, $Q = U\left(\CP{n-1}{T}\right) + \bm{b}^{(n)}$. We let $\bm{x}^{(n)} \in \CP{n-1}{T}$ be arbitrary. Then
    \begin{align*}
        \bm{x}^{(n)} = \sum_{i = 1}^{2^{n - 2}}\kappa_i\left(\bm{s}_i^{(n - 1)}, 1\right) + \sum_{i = 1}^{2^{n - 2}}\lambda_i\left(\bm{s}_i^{(n - 1)} + \bm{e}_{n - 1}^{(n - 1)}, 0\right),
    \end{align*}
    where $0 \le \kappa_i, \lambda_i$ for all $i \in \mathbb{N}_1^{2^{n - 2}}$ and
    \begin{align*}
        \sum_{i = 1}^{2^{n - 2}}\kappa_i + \sum_{i = 1}^{2^{n - 2}}\lambda_i = 1.
    \end{align*}
    We have that
   \begin{equation}\label{eq6.2}
        \begin{split}
            U\bm{x}^{(n)} &= \sum_{i = 1}^{2^{n - 2}}\kappa_iU\left(\bm{s}_i^{(n - 1)}, 1\right) + \sum_{i = 1}^{2^{n - 2}}\lambda_iU\left(\bm{s}_i^{(n - 1)} + \bm{e}_{n - 1}^{(n - 1)}, 0\right) \\
            &= \sum_{i = 1}^{2^{n - 2}}\kappa_i\left(\bm{s}_i^{(n - 1)} + \bm{e}_{n - 1}^{(n - 1)}, 1\right) + \sum_{i = 1}^{2^{n - 2}}\lambda_i\left(\bm{s}_i^{(n - 1)} + \bm{e}_{n - 1}^{(n - 1)}, 0\right) \\
            &= \sum_{i = 1}^{2^{n - 2}}\kappa_i\left(\bm{s}_i^{(n - 1)}, 1\right) + \sum_{i = 1}^{2^{n - 2}}\lambda_i\left(\bm{s}_i^{(n - 1)}, 0\right) + \bm{e}_{n - 1}^{(n)}.
        \end{split}
    \end{equation}
%
%
%
    Therefore,
    \begin{align}\label{eq6.3}
        U\bm{x}^{(n)} + \bm{b}^{(n)} = \sum_{i = 1}^{2^{n - 2}}\kappa_i\left(\bm{s}_i^{(n - 1)}, 1\right) + \sum_{i = 1}^{2^{n - 2}}\lambda_i\left(\bm{s}_i^{(n - 1)}, 0\right).
    \end{align}
    Note that the right-hand side of \eqref{eq6.2} has the same form as the right-hand side of \eqref{eq6.1}, which is an element of $Q$. Thus, $U\bm{x}^{(n)} + \bm{b}^{(n)} \in Q$, which implies that $U\left(\CP{n-1}{T}\right) + \bm{b}^{(n)} \subseteq Q$. We now let
    \begin{align*}
        \bm{y}^{(n)} = \sum_{i = 1}^{2^{n - 2}}\kappa_i\left(\bm{s}_i^{(n - 1)}, 1\right) + \sum_{i = 1}^{2^{n - 2}}\lambda_i\left(\bm{s}_i^{(n - 1)}, 0\right) \in Q
    \end{align*}
    be arbitrary as in \eqref{eq6.1}. Then by \eqref{eq6.2} and \eqref{eq6.3}, we have that for 
    \begin{align*}
        \bm{x}^{(n)} = \sum_{i = 1}^{2^{n - 2}}\kappa_i\left(\bm{s}_i^{(n - 1)}, 1\right) + \sum_{i = 1}^{2^{n - 2}}\lambda_i\left(\bm{s}_i^{(n - 1)} + \bm{e}_{n - 1}^{(n - 1)}, 0\right) \in \CP{n-1}{T},
    \end{align*}
    $\bm{y}^{(n)} = U\bm{x}^{(n)} + \bm{b}^{(n)}$. This implies that $Q \subseteq U\left(\CP{n-1}{T}\right) + \bm{b}^{(n)}$. Thus, we obtain $Q = U\left(\CP{n-1}{T}\right) + \bm{b}^{(n)}$, as desired.
\end{proof}

We conclude this section by noting that Theorem \ref{cubeunimodularequivalence} implies that $\CP{n-1}{T}$ has the integer decomposition property (c.f. \cite{BeckRobins}). 
Additionally, since the unit cube is a $0/1$ polytope, we ascertain the vertex structure of $\CP{n-1}{T}$, namely that every minimally self-reachable configuration is a vertex.
In general, proving the integer decomposition property and describing the vertex structure is more involved and will be discussed in Sections \ref{sec:IDP} and \ref{sec:vertices} respectively.





\section{The Integer Decomposition Property}\label{sec:IDP}

In this section, we study the integer decomposition property for $\CP{\ell}{T}$.
In particular, we will build towards the following theorem. 

\begin{theorem}\label{thm:IDP}
   For any $n$-vertex tree $T$ and any $\ell\in\mathbb{N}_{n-1}$, the polytope $\CP{\ell}{T}$ has the integer decomposition property.
\end{theorem}

To begin, we must introduce some new terminology. 
Given an $n$-vertex tree $T$ equipped with chip configuration $\bm{s}^{(n)}$ and $t\in \mathbb{N}_1$, we say that $\bm{s}^{(n)}$ is a \Def{$t$-dilated configuration} on $T$ if $\bm{s}^{(n)}$ contains at least $t(m - 1)$ chips on every $m$-vertex subtree of $T$. Observe that $\bm{s}^{(n)}$ is self-reachable if and only if $\bm{s}^{(n)}$ is 1-dilated by Theorem \ref{srcsubtreeresult}.
We say that $\bm{s}^{(n)}$ is \Def{minimally $t$-dilated} on $T$ if $\bm{s}^{(n)}$ is $t$-dilated on $T$ and has exactly $t(n - 1)$ chips on $T$.
The inspiration for the name $t$-dilated derives from the observation of the following lemma. 


\begin{lemma}\label{chipconfigdilation}
    Let $T$ be an $n$-vertex tree. For each $t \in \mathbb{N}_1$ and for all $\bm{x}^{(n)} \in t \cdot \CP{\ell}{T} \cap \mathbb{Z}^n$, $\bm{x}^{(n)}$ is a $t$-dilated configuration on $T$. 
\end{lemma}

\begin{proof}
	Let $\CP{\ell}{T} = \left\{\bm{s}_1^{(n)}, \bm{s}_2^{(n)}, \ldots, \bm{s}_k^{(n)}\right\}$. By definition,    
    \begin{equation}
        \bm{x}^{(n)} = \sum_{i = 1}^k\lambda_i\left(t\bm{s}_i^{(n)}\right) = t\sum_{i = 1}^k\lambda_i\bm{s}_i^{(n)}\notag
    \end{equation}
    with $\lambda_i \ge 0$ for all $i \in \mathbb{N}_{[1,n]}$ and $\sum_{i = 1}^k\lambda_i = 1$. Now choose an arbitrary subtree of $T$ with $m$ vertices and let $\bm{\chi}^{(n)}$ be the characteristic vector of this subtree. Then by Theorem \ref{srcsubtreeresult},
    \begin{equation}
        \bm{\chi}^{(n)} \cdot \bm{x}^{(n)} = t\sum_{i = 1}^k\lambda_i\left(\bm{\chi}^{(n)} \cdot \bm{s}_i^{(n)}\right) \ge t\sum_{i = 1}^k\lambda_i\left(m - 1\right) = t(m - 1).\notag
    \end{equation}
\end{proof}

The special case of $t=1$ in Lemma \ref{chipconfigdilation} yields the following corollary that all lattice points with in $\CP{\ell}{T}$ are in fact self-reachable configurations on $T$. 

\begin{corollary}\label{SRC Convex Combination Closure}
    Let $T$ be an $n$-vertex tree. Then for all $\bm{x}^{(n)} \in \CP{\ell}{T} \cap \mathbb{Z}^n$, $\bm{x}^{(n)}$ is self-reachable on $T$.
\end{corollary}

We now proceed to show a collection of necessary lemmata on $t$-dilated configurations in the eventual service of Theorem \ref{thm:IDP}.
To begin, we show that the property of minimally $t$-dialted subtrees extends to intersections and unions.

\begin{lemma}\label{combineminimalsrcs}
    Let $T$ be an $n$-vertex tree, and let $\bm{s}^{(n)}$ be a $t$-dilated configuration on $T$ for some $t \in \mathbb{N}_1$. Suppose $T_1$ and $T_2$ are subtrees of $T$ that satisfy the following properties:
    \begin{enumerate}
        \item $T_1$ and $T_2$ have at least one vertex in common, i. e. $T_1 \cap T_2 \neq \varnothing$, and
        \item $\bm{s}^{(n)}$ forms a minimally $t$-dilated configuration on both $T_1$ and $T_2$.
    \end{enumerate}
    Then $\bm{s}^{(n)}$ forms a minimally $t$-dilated configuration on both $T_1 \cap T_2$ and $T_1 \cup T_2$.
\end{lemma}

\begin{proof}
    Let $m_1=|V(T_1)|$, $m_2=|V(T_2)|$, $m_3=|V(T_1\cap T_2)|$, and $M=|V(T_1\cup T_2)|$.
    Additionally, let $\ell$ be the number of chips $\bm{s}^{(n)}$ has on $T_1 \cap T_2$, and $L$ be the number of chips $\bm{s}^{(n)}$ has on $T_1 \cup T_2$.
    %
    
    Our goal is to show that $\ell = t(m_3 - 1)$ and $L = t(M - 1)$. 
    It is immediately clear that $M=m_1+m_2-m_3$ and by property (2) above we also have
    %
    %
    \begin{equation}\label{eq30}
        L = t(m_1 - 1) + t(m_2 - 1) - \ell.
    \end{equation}
    Because $\bm{s}^{(n)}$ is $t$-dilated on $T$, $L \ge t(M - 1) = t(m_1 + m_2 - m_3 - 1)$ and $\ell \ge t(m_3 - 1)$. This, combined with \eqref{eq30}, implies that
    \begin{equation}
        \begin{split}
            t(m_1 + m_2 - m_3 - 1) &\le L = t(m_1 - 1) + t(m_2 - 1) - \ell \\
            &\le t(m_1 - 1) + t(m_2 - 1) - t(m_3 - 1) \\
            &= t(m_1 + m_2 - m_3 - 1).
        \end{split}
    \end{equation}
    Therefore, $t(m_1 + m_2 - m_3 - 1) = t(m_1 - 1) + t(m_2 - 1) - \ell$, so $\ell = t(m_3 - 1)$ and $L = t(m_1 - 1) + t(m_2 - 1) - t(m_3 - 1) = t(m_1 + m_2 - m_3 - 1) = t(M - 1)$. 
\end{proof}

The following lemma gives sufficient conditions for determining if a $t$-dilated configuration is minimally $t$-dilated. 

\begin{lemma}\label{Deduce t-dilated}
     Let $T$ be an $n$-vertex tree, and let $\bm{s}^{(n)}$ be a $t$-dilated configuration on $T$ for some $t \in \mathbb{N}_1$. Suppose there exists a set $\mathcal{T} = \{T_1, T_2, \ldots, T_k\}$ of subtrees of $T$ with the following properties:

     \begin{enumerate}
         \item Each vertex of $T$ is contained in $T_i$ for some $i \in \mathbb{N}_{[1,k]}$.

         \item $\bm{s}^{(n)}$ forms a minimally $t$-dilated configuration on $T_i$ for all $i \in \mathbb{N}_{[1,k]}$.
     \end{enumerate}
     
     Then $\bm{s}^{(n)}$ is minimally $t$-dilated on $T$.
\end{lemma}

\begin{proof}
	For now, assume that  $T_{k - 1}$ and $T_k$ have at least one vertex in common. 
    By Lemma \ref{combineminimalsrcs}, $\bm{s}^{(n)}$ forms a minimally $t$-dilated configuration on $T_{k - 1} \cup T_k$. Then $\{T_1, T_2, \ldots, T_{k - 2}, T_{k - 1} \cup T_k\}$ is a set of $k - 1$ trees that satisfies the two properties stated above. We can repeat this process until we obtain a set of non-intersecting trees that satisfies the two properties stated above. So, assume that all of the trees in $\mathcal{T}$ are non-intersecting. Note that if $k \ge 2$, then there exist neighboring vertices $v_a$ and $v_b$ that are contained in different trees in $\mathcal{T}$. Without loss of generality, we can relabel the vertices of $T$ and the trees in $\mathcal{T}$ such that $a = 1$, $b = 2$, $T_1$ contains $v_1$, and $T_2$ contains $v_2$. We let $m_1$ and $m_2$ be the number of vertices on $T_1$ and $T_2$, respectively. Since $\bm{s}^{(n)}$ forms a minimally $t$-dilated configuration on both $T_1$ and $T_2$, $\bm{s}^{(n)}$ has exactly $t(m_1 - 1)$ chips on $v_1$ and $t(m_2 - 1)$ chips on $v_2$. We now let $T^*$ be the tree formed by connecting $T_1$ and $T_2$ by the edge between $v_1$ and $v_2$. Then $\bm{s}^{(n)}$ has exactly $t(m_1 + m_2 - 2)$ chips on $T^*$. But $T^*$ has $m_1 + m_2$ vertices, which contradicts the assumption that $\bm{s}^{(n)}$ is $t$-dilated on $T$.   
\end{proof}

The following lemma ensures that one can always remove a chip from a non-minimally $t$-dilated configuration so that the resulting configuration remains $t$-dilated. 

\begin{lemma}\label{subtractoffchips}
    Let $T$ be an $n$-vertex tree, and let $\bm{s}^{(n)}$ be a non-minimally $t$-dilated configuration on $T$ for some $t \in \mathbb{N}_1$. Then there exists $i \in \mathbb{N}_{[1,n]}$ such that $\bm{s}^{(n)} - \bm{e}_i^{(n)}$ is a $t$-dilated configuration on $T$.
\end{lemma}

\begin{proof}
  Proceeding by contrapositive, we show that if $\bm{s}^{(n)}$ is $t$-dilated but that for all $i \in \mathbb{N}_1^n$, $\bm{s}^{(n)} - \bm{e}_i^{(n)}$ is not $t$-dilated, then $\bm{s}^{(n)}$ must be minimally $t$-dilated. Note that if $\bm{s}^{(n)}$ is $t$-dilated on $T$ but $\bm{s}^{(n)} - \bm{e}_i^{(n)}$ is not $t$-dilated on $T$, then there must exist an $m$-vertex subtree $T^*$ of $T$ containing $v_i$ such that $\bm{s}^{(n)}$ has exactly $t(m - 1)$ chips on $T^*$. Since such a subtree must exist for each $i \in \mathbb{N}_1^n$, we can deduce the existence of a set $\mathcal{T}$ of subtrees of $T$ that satisfy the conditions stated in Lemma \ref{Deduce t-dilated} and thus we conclude that $\bm{s}^{(n)}$ must be minimally $t$-dilated on $T$.

\end{proof}




The following lemma shows that one can move from a $(t+1)$-dilated configuration to a $t$-dialted configuration via subtraction of a minimally self-reachable configuration.

\begin{lemma}\label{theorem6.8}
    Let $T$ be an $n$-vertex tree, and let $\bm{c}^{(n)}$ be a $(t + 1)$-dilated configuration on $T$ for some $t \in \mathbb{N}_1$. Then there exists a minimally self-reachable configuration $\bm{s}^{(n)}$ on $T$ such that $\bm{c}^{(n)} - \bm{s}^{(n)}$ is $t$-dilated on $T$.
\end{lemma}

\begin{proof}
    We proceed by induction on $n$. 
    For the base case of $n=1$, note that the chip configuration with 0 chips on $v_1$ is a minimally self-reachable configuration. 
	Additionally, note that if $\bm{c}^{(1)}$ is $(t + 1)$-dilated, then $\bm{c}^{(1)}$ is $t$-dilated. Thus, $\bm{c}^{(1)} - \bm{0}^{(1)}$ is a $t$-dilated configuration on $T$. 
	Now for $k\geq 1$, assume that for any $k$-vertex tree $T$ and any $(t + 1)$-dilated configuration $\bm{c}^{(k)}$ on $T$, there exists a minimally self-reachable configuration $\bm{s}^{(k)}$ on $T$ such that $\bm{c}^{(k)} - \bm{s}^{(k)}$ is $t$-dilated on $T$. We now let $T$ be a tree with $k + 1$ vertices such that $v_{k + 1}$ is a leaf connected to $v_k$. We let $\left(\bm{c}^{(k)}, c_{k + 1}\right)$ be a $(t + 1)$-dilated chip configuration on $T$ and consider the following cases:
    
 \textbf{Case 1:} $c_{k + 1} \ge t + 1$. Note that $\bm{c}^{(k)}$ must be $(t + 1)$-dilated on $T \setminus v_{k + 1}$, so there exists a minimally self-reachable configuration $\bm{s}^{(k)}$ on $T \setminus v_{k + 1}$ such that $\bm{c}^{(k)} - \bm{s}^{(k)}$ is $t$-dilated on $T \setminus v_{k + 1}$. Then $\left(\bm{s}^{(k)}, 1\right)$ is a minimally self-reachable configuration on $T$ by Corollary \ref{AddChipToLeafOrNeighbor}. Now suppose towards contradiction that $\left(\bm{c}^{(k)}, c_{k + 1}\right) - \left(\bm{s}^{(k)}, 1\right)$ is not $t$-dilated on $T$. Then there exists an $m$-vertex subtree $T^*$ of $T$ such that $\left(\bm{c}^{(k)}, c_{k + 1}\right) - \left(\bm{s}^{(k)}, 1\right)$ has fewer than $t(m - 1)$ chips on $T^*$. $T^*$ must include $v_{k + 1}$ since $\bm{c}^{(k)} - \bm{s}^{(k)}$ is $t$-dilated on $T \setminus v_{k + 1}$. Also note that since $c_{k + 1} - 1 \ge t$, the configuration $\left(\bm{c}^{(k)}, c_{k + 1}\right) - \left(\bm{s}^{(k)}, 1\right)$ must have at least $t$ chips on $v_{k + 1}$. This implies that $\left(\bm{c}^{(k)}, c_{k + 1}\right) - \left(\bm{s}^{(k)}, 1\right)$ must have fewer than $t(m - 1) - t = t(m - 2)$ chips on the $(m - 1)$-vertex tree $T^* \setminus v_{k + 1}$. But this contradicts our assumption that $\bm{c}^{(k)} - \bm{s}^{(k)}$ is $t$-dilated on $T \setminus v_{k + 1}$. Thus, $\left(\bm{c}^{(k)}, c_{k + 1}\right) - \left(\bm{s}^{(k)}, 1\right)$ is $t$-dilated on $T$.

\textbf{Case 2:} $c_{k + 1} \le t$. Suppose towards contradiction that $\bm{c}^{(k)} - (t + 1 - c_{k + 1})\bm{e}_k^{(k)}$ is not $(t + 1)$-dilated on $T \setminus v_{k + 1}$. Then for some $m$-vertex subtree $T^*$ of $T \setminus v_{k + 1}$ that contains $v_k$, $\bm{c}^{(k)} - (t + 1 - c_{k + 1})\bm{e}_k^{(k)}$ has fewer than $(t + 1)(m - 1)$ chips on $T^*$. Therefore, $\bm{c}^{(k)}$ has fewer than $(t + 1)(m - 1) + t + 1 - c_{k + 1} = (t + 1)m - c_{k + 1}$ chips on $T^*$. If we let $T^{**}$ be the subtree of $T$ obtained by attaching $v_{k + 1}$ to $T^*$ at $v_k$, then $\left(\bm{c}^{(k)}, c_{k + 1}\right)$ has fewer than $(t + 1)m$ chips on $T^{**}$, which contradicts our assumption that $\left(\bm{c}^{(k)}, c_{k + 1}\right)$ is $t$-dilated on $T$. Therefore, $\bm{c}^{(k)} - (t + 1 - c_{k + 1})\bm{e}_k^{(k)}$ must be $(t + 1)$-dilated on $T \setminus v_{k + 1}$. Hence, there exists a minimally self-reachable configuration $\bm{s}^{(k)}$ on $T \setminus v_{k + 1}$ such that $\bm{c}^{(k)} - (t + 1 - c_{k + 1})\bm{e}_k^{(k)} - \bm{s}^{(k)}$ is $t$-dilated on $T \setminus v_{k + 1}$. Then $\left(\bm{s}^{(k)} + \bm{e}_k^{(k)}, 0\right)$ is a minimally self-reachable configuration on $T$ by Corollary \ref{AddChipToLeafOrNeighbor}. We suppose towards contradiction that
        \begin{equation}
            \left(\bm{c}^{(k)}, c_{k + 1}\right) - \left(\bm{s}^{(k)} + \bm{e}_k^{(k)}, 0\right)\notag
        \end{equation}
        is not $t$-dilated on $T$. Then for some $m$-vertex subtree $T^*$ of $T$, $\left(\bm{c}^{(k)}, c_{k + 1}\right) - \left(\bm{s}^{(k)} + \bm{e}_k^{(k)}, 0\right)$ has fewer than $t(m - 1)$ chips. Note that since $c_{k + 1} \le t$, $\bm{c}^{(k)} - (t + 1 - c_{k + 1})\bm{e}_k^{(k)} - \bm{s}^{(k)}$ has at most as many chips on $v_k$ as $\bm{c}^{(k)} - \bm{s}^{(k)} - \bm{e}_k^{(k)}$. Therefore, $T^*$ must contain $v_{k + 1}$ or else $\bm{c}^{(k)} - (t + 1 - c_{k + 1})\bm{e}_k^{(k)} - \bm{s}^{(k)}$ couldn't be $t$-dilated on $T \setminus v_{k + 1}$. Now note that $\left(\bm{c}^{(k)}, c_{k + 1}\right) - \left(\bm{s}^{(k)} + \bm{e}_k^{(k)}, 0\right)$ contains fewer than $t(m - 1) - c_{k + 1}$ chips on $T^* \setminus v_{k + 1}$. So $\bm{c}^{(k)} - (t + 1 - c_{k + 1})\bm{e}_k^{(k)} - \bm{s}^{(k)}$ must contain fewer than $t(m - 1) - t = t(m - 2)$ chips on $T^* \setminus v_{k + 1}$. But since $T^* \setminus v_{k + 1}$ has $m - 1$ vertices, this contradicts the fact that $\bm{c}^{(k)} - (t + 1 - c_{k + 1})\bm{e}_k^{(k)} - \bm{s}^{(k)}$ is $t$-dilated on $T \setminus v_{k + 1}$. Thus, $\left(\bm{c}^{(k)}, c_{k + 1}\right) - \left(\bm{s}^{(k)} + \bm{e}_k^{(k)}, 0\right)$ is $t$-dilated on $T$.
        
    This completes the proof.
\end{proof}

The following technical lemma descends from Lemma \ref{subtractoffchips} and Lemma \ref{theorem6.8}.

\begin{lemma}\label{cor6.9}
    Let $T$ be an $n$-vertex tree, and let $\bm{c}^{(n)}$ be a $(t + 1)$-dilated self-reachable configuration with $\ell$ chips on $T$ for some $t \in \mathbb{N}_1$. 
	Then for any $n-1\leq L \leq \ell - t(n - 1)$, there exists a self-reachable configuration $\bm{s}^{(n)}$ on $T$ with $L$ chips such that $\bm{c}^{(n)} - \bm{s}^{(n)}$ is $t$-dilated on $T$.   
\end{lemma}

\begin{proof}
    We let $L$ be an arbitrary integer between $n - 1$ and $\ell - t(n - 1)$. By Lemma \ref{theorem6.8}, there exists a minimally self-reachable configuration $\bm{s}^{(n)}$ such that $\bm{c}^{(n)} - \bm{s}^{(n)}$ is $t$-dilated on $T$. Since $\bm{c}^{(n)}$ has $\ell \ge (t + 1)(n - 1)$ chips on $T$, it will always be possible to remove $L$ chips from the configuration $\bm{c}^{(n)}$ while keeping at least $t(n - 1)$ chips on $T$. So since $\bm{s}^{(n)}$ must have $n - 1$ chips on $T$, we know that we can remove up to $L - (n - 1)$ more chips from the configuration $\bm{c}^{(n)} - \bm{s}^{(n)}$ while keeping at least $t(n - 1)$ chips on $T$. Therefore, by Lemma \ref{subtractoffchips}, we can remove $L - (n - 1)$ more chips from the configuration $\bm{c}^{(n)} - \bm{s}^{(n)}$ while maintaining a $t$-dilated configuration on $T$. We let $\bm{x}^{(n)}$ be the number of chips removed from each vertex starting from the configuration $\bm{c}^{(n)} - \bm{s}^{(n)}$. Then $\bm{s}^{(n)} + \bm{x}^{(n)}$ is a self-reachable configuration on $T$ with $L$ chips by Proposition \ref{Add Chips to SRC} since all of the components of $\bm{x}^{(n)}$ are nonnegative, and we have that $\bm{c}^{(n)} - \left(\bm{s}^{(n)} + \bm{x}^{(n)}\right)$ is $t$-dilated on $T$, as desired.
    %
    %
\end{proof}


We have now built the necessary machinery to prove Theorem \ref{thm:IDP}.

\begin{proof}[Proof of Theorem~\ref{thm:IDP}]
	To show that $\CP{\ell}{T}$ has the integer decomposition property, we first show that for $t\in\mathbb{N}_1$, $t\cdot \CP{\ell}{T}$ is the convex hull of all $t$-dilated configurations with $t\ell$ chips. 
    To do so, it suffices to show that the lattice points of $t \cdot \CP{\ell}{T}$ are precisely the $t$-dilated configurations on $T$ with $t\ell$ chips. 
    Lemma \ref{chipconfigdilation} implies that all of the lattice points of $t\cdot \CP{\ell}{T}$ are $t$-dilated configurations on $T$. 
    Furthermore, these configurations must have $t\ell$ chips since all of the configurations in $S_\ell^{(T)}$ have $\ell$ chips. 
    Therefore, we only need to show that all of the $t$-dilated configurations with $t\ell$ chips belong to $t \cdot \CP{\ell}{T}$. 
    This is immediate for $t=1$, so consider $t\geq 2$.
    Given a $t$-dilated chip configuration $\bm{c}^{(n)}$ with $t\ell$ chips for $t \in \mathbb{N}_2$, note that $n - 1 \le \ell \le t\ell - (t - 1)(n - 1)$ since $\ell \ge n - 1$. 
     We can then apply Lemma \ref{cor6.9} to obtain a self-reachable configuration $\bm{s}_1^{(n)}$ on $T$ with $\ell$ chips such that $\bm{c}^{(n)} - \bm{s}_1^{(n)}$ is $(t - 1)$-dilated on $T$. 
    Then if $t \ge 3$, then $n - 1 \le \ell \le (t - 1)\ell - (t - 2)(n - 1)$, so we can apply Lemma \ref{cor6.9} again to deduce the existence of a self-reachable configuration $\bm{s}_2^{(n)}$ on $T$ with $\ell$ chips such that $\bm{c}^{(n)} - \bm{s}_1^{(n)} - \bm{s}_2^{(n)}$ is $(t - 2)$-dilated on $T$. 
    Iteratively applying Lemma \ref{cor6.9} $t - 1$ times deduces the existence of self-reachable configurations $\bm{s}_1^{(n)}, \bm{s}_2^{(n)}, \ldots, \bm{s}_{t - 1}^{(n)}$ on $T$ with $\ell$ chips such that
    \begin{align*}
        \bm{c}^{(n)} - \sum_{j = 1}^{t - 1}\bm{s}_j^{(n)}
    \end{align*}
    is 1-dilated, i.e. self-reachable, on $T$. This configuration must have $\ell$ chips on $T$ since $\bm{c}^{(n)}$ has $t\ell$ chips on $T$. If we let
    \begin{align*}
        \bm{s}_t^{(n)} = \bm{c}^{(n)} - \sum_{j = 1}^{t - 1}\bm{s}_j^{(n)},
    \end{align*}
    we have that
    \begin{align}\label{eq31}
        \bm{c}^{(n)} = \sum_{j = 1}^{t}\bm{s}_j^{(n)}.
    \end{align}
    where $\bm{s}_j^{(n)}$ is a self-reachable configuration with $\ell$ chips for all $j \in \mathbb{N}_{[1,t]}$, meaning $\bm{s}_j^{(n)} \in  \CP{\ell}{T}$ for all $j \in \mathbb{N}_{[1,t]}$. Now note that we can recast \eqref{eq31} as
    \begin{equation}\label{eq32}
        \bm{c}^{(n)} = \sum_{j = 1}^{t}\frac{1}{t}\left(t\bm{s}_j^{(n)}\right).
    \end{equation}
    Since \eqref{eq32} writes $\bm{c}^{(n)}$ as a convex combination of $t\bm{s}_j^{(n)} \in t\cdot \CP{\ell}{T}$, we have that $\bm{c}^{(n)}$ is a lattice point of $t\cdot  \CP{\ell}{T}$, as desired. Therefore, the lattice points of $t\cdot  \CP{\ell}{T}$  are precisely the $t$-dilated configurations on $T$ with $t\ell$ chips.
    Furthermore, \eqref{eq31} shows that we can express any such lattice point as a sum of $t$ lattice points in $S_\ell^{(T)}$ proving that $ \CP{\ell}{T}$ has the integer decomposition property.
\end{proof}



\section{Vertex Structure of Self-Reachable Configuration Polytopes}\label{sec:vertices}

In this section, we explore the vertex structure of $\CP{\ell}{T}$ in the case where $\ell\geq n$. 
It is clear from Corollary \ref{SRC Convex Combination Closure} that every lattice point within $\CP{\ell}{T}$ is a self-reachable configuration on $T$, but what is unclear as of yet is which configurations form the extreme points. 
In order to answer this question, we need a definition.

\begin{definition}
    Let $T$ be an $n$-vertex tree, and let $\bm{\nu}^{(n)}$ be a non-minimally self-reachable configuration on $T$. We say that $\bm{\nu}^{(n)}$ is a \Def{near-minimally self-reachable configuration} on $T$ if there exists a unique $i \in \mathbb{N}_{[1,n]}$ such that $\bm{\nu}^{(n)} - \bm{e}_i^{(n)}$ is also self-reachable on $T$.
     For this $i \in \mathbb{N}_{[1,n]}$, we say that $\bm{\nu}^{(n)}$ is \Def{near-minimally self-reachable about $v_i$}.
\end{definition}

\begin{remark}
It is worth noting that the existence of such an $i$ follows from the $t = 1$ case of Lemma \ref{subtractoffchips}.
\end{remark}

With this definition at hand, we can now describe the vertex structure of $\CP{\ell}{T}$ via the following theorem. 

\begin{theorem}\label{thm:vertexStructure}
    Let $T$ be an $n$-vertex tree, and let $\ell \in \mathbb{N}_n$. The vertices of $\CP{\ell}{T}$ are precisely the near-minimally self-reachable configurations on $T$ with $\ell$ chips. 
\end{theorem}

Before we can produce a proof of this main result, it is necessary to explore properties of near-minimally self-reachable configurations via some lemmata. 
To this end, we begin with following result characterizing such configurations. 

\begin{lemma}\label{Near-Minimal Properties Lemma}
    Let $T$ be an $n$-vertex tree, and let $\bm{\nu}^{(n)}$ be a self-reachable configuration on $T$. Then $\bm{\nu}^{(n)}$ is near-minimally self-reachable about $v_t$ if and only if:

    \begin{enumerate}
        \item $\bm{\nu}^{(n)}$ forms a minimally self-reachable configuration on each tree in $T \setminus v_t$, and

        \item $\bm{\nu}^{(n)}$ has at least $d + 1$ chips on $v_t$, where $d = \deg(v_t)$.
    \end{enumerate}
\end{lemma}

\begin{proof}
    Without loss of generality, we suppose that $t = n$. 
    We let $T_1, T_2, \ldots, T_d$ be the trees in $T \setminus v_n$. 
    For sufficiency, assume that $\bm{\nu}^{(n)}$ is near-minimally self-reachable about $v_n$. 
    Let $T^*$ be an arbitrary tree in $T \setminus v_n$, and let $v_1, v_2, \ldots, v_m$ be the vertices of $T^*$. 
    Because $\bm{\nu}^{(n)} - \bm{e}_j^{(n)}$ cannot be self-reachable on $T$ for all $j \in \mathbb{N}_{[1,m]}$, there must exist a set $\mathscr{T} = \{T_1^*, T_2^*, \ldots, T_k^*\}$ of subtrees of $T$ such that

    \begin{itemize}
        \item Each vertex of $T^*$ is contained in $T_i^*$ for some $i \in \mathbb{N}_{[1,k]}$, and

        \item $\bm{\nu}^{(n)}$ forms a minimally self-reachable configuration on $T_i^*$ for all $i \in \mathbb{N}_{[1,k]}$.
    \end{itemize}

    Note that none of the subtrees of $T$ in $\mathscr{T}$ can include $v_n$ since $\bm{\nu}^{(n)} - \bm{e}_n^{(n)}$ is self-reachable on $T$. 
    Therefore, $\mathscr{T}$ is in fact a set of subtrees of $T^*$. 
    The $t = 1$ case of Lemma \ref{Deduce t-dilated} implies that $\bm{\nu}^{(n)}$ must form a minimally self-reachable configuration on $T^*$. 
    Since $T^*$ was arbitrarily selected, we conclude that $\bm{\nu}^{(n)}$ forms a minimally self-reachable configuration on each tree in $T \setminus v_i$. So if we let $m_1, m_2, \ldots, m_d$ be the number of vertices on the trees in $T \setminus v_n$, then $\bm{\nu}^{(n)}$ must have exactly
    \begin{align*}
        \sum_{j = 1}^{d}(m_j - 1) = n - 1 - d
    \end{align*}
    chips on these trees in total. Since $\bm{\nu}^{(n)}$ must have at least $n$ chips on $T$ to avoid being non-self-reachable or minimally self-reachable, $\bm{\nu}^{(n)}$ must have at least $d + 1$ chips on $v_n$. 
    
    To show necessity, note that property (1) guarantees that $\bm{\nu}^{(n)} - \bm{e}_j^{(n)}$ is not self-reachable for $j \in \mathbb{N}_{[1,n-1]}$. 
    Let $T^*$ be an arbitrary $m$-vertex subtree of $T$. We will show that $\bm{\nu}^{(n)} - \bm{e}_n^{(n)}$ has at least $m - 1$ chips on $T^*$. 
    Since $\bm{\nu}^{(n)}$ is self-reachable on $T$, the desired claim follows immediately unless $T^*$ contains $v_n$. In this case, for each $i \in \mathbb{N}_{[1,d]}$, we let $M_i$ be the number of vertices $T^*$ has on $T_i$.
     Note that because $\bm{\nu}^{(n)}$ is self-reachable on $T$, $\bm{\nu}^{(n)} - \bm{e}_n^{(n)}$ has at least $M_i - 1$ chips on the branch of $T^*$ that overlaps with $T_i$. 
    Also, $\bm{\nu}^{(n)} - \bm{e}_n^{(n)}$ has at least $d$ chips on $v_n$ by property (2) of $\bm{\nu}^{(n)}$. This means that $\bm{\nu}^{(n)} - \bm{e}_n^{(n)}$ has at least
    \begin{align*}
        d + \sum_{i = 1}^{d}(M_i - 1) = \sum_{i = 1}^{d}M_i = m - 1
    \end{align*}
    chips on $T^*$, as desired. This completes the proof.
\end{proof}

The following lemma provides another useful characterization of near-minimally self-reachable configurations. Specifically, this result specifies that there must be a vertex with the maximum possible chips to maintain self-reachability. 

\begin{lemma}\label{Max Chips Near-Minimal}
    Let $T$ be an $n$-vertex tree. For any vertex $v_i$ of $T$, any $\ell \in \mathbb{N}_n$, and any self-reachable configuration $\bm{\nu}^{(n)}$ on $T$ with $\ell$ chips, $\bm{\nu}^{(n)}$ is a near-minimally self-reachable configuration on $T$ about $v_t$ if and only if $\bm{\nu}^{(n)}$ has exactly $\ell + d + 1 - n$ of chips on $v_t$ (the maximum possible number given that $\bm{\nu}^{(n)}$ is self-reachable on $T$), where $d = \deg(v_t)$.
\end{lemma}

\begin{proof}
    Without loss of generality, we suppose that $t = n$. 
    We let $T_1, T_2, \ldots, T_d$ be the trees in $T \setminus v_n$, and we let $m_j$ be the number of chips on $T_j$ for each $j \in \mathbb{N}_{[1,m]}$. 
    Note that since $\bm{\nu}^{(n)}$ is self-reachable, $\bm{\nu}^{(n)}$ must have at least $m_j - 1$ chips on $T_j$ for each $j \in \mathbb{N}_{[1,m]}$. 
    To prove sufficiency, we assume that $\bm{\nu}^{(n)}$ is a near-minimally self-reachable configuration on $T$ about $v_n$ with $\ell$ chips. 
    In the proof of Lemma \ref{Near-Minimal Properties Lemma}, we saw that $\bm{\nu}^{(n)}$ must have exactly 
    \begin{align*}
        \sum_{j = 1}^d(m_j - 1) = n - 1 - d
    \end{align*}
    chips on the vertices excluding $v_n$. Therefore, there must be $\ell - (n - 1 - d) = \ell + d + 1 - n$ chips on $v_n$. 
    We cannot have fewer than $n - 1 - d$ chips on the trees in $T \setminus v_n$ or else we couldn't have $m_j - 1$ chips on $T_j$ for each $j \in \mathbb{N}_{[1,n]}$. 
    Therefore, $\bm{\nu}^{(n)}$ can have a maximum of $\ell + d + 1 - n$ chips on $v_n$. 
    
    To prove necessity, we assume that $\bm{\nu}^{(n)}$ has exactly $\ell + d + 1 - n$ chips on $v_n$. Then $\bm{\nu}^{(n)}$ has exactly
    \begin{align*}
        n - 1 - d = \sum_{j = 1}^d(m_j - 1)
    \end{align*}
    total chips on the trees $T_1, T_2, \ldots, T_d$. Because $\bm{\nu}^{(n)}$ is self-reachable on $T$, this means that $\bm{\nu}^{(n)}$ must have $m_j - 1$ chips on $T_j$ and form a minimally self-reachable configuration on $T_j$ for each $j \in \mathbb{N}_{[1,n]}$. So $\bm{\nu}^{(n)}$ forms a minimally self-reachable configuration on each tree in $T \setminus v_n$, and $\bm{\nu}^{(n)}$ has at least $d + 1$ chips on $v_n$ since $\ell \ge n$. So by Lemma \ref{Near-Minimal Properties Lemma}, $\bm{\nu}^{(n)}$ is a near-minimally self-reachable configuration on $T$ about $v_n$. This completes the proof.
\end{proof}

Given a near-minimally self-reachable configuration about a vertex $v_t$, one preserves this property when adjusting the configuration by adding chips to $v_t$.
This is verified via the following lemma.

\begin{lemma}\label{Add Chips to Near-Minimal Configuration}
    Let $T$ be an $n$-vertex tree. Suppose that $\bm{\nu}^{(n)}$ is a near-minimally self-reachable configuration on $T$ about $v_t$ with $\ell$ chips for some $t \in \mathbb{N}_{[1,n]}$ and $\ell \in \mathbb{N}_n$. Then $\bm{\nu}^{(n)} + L\bm{e}_t^{(n)}$ is a near-minimally self-reachable configuration on $T$ about $v_t$ with $\ell + L$ chips for any $L \in \mathbb{N}_0$.
\end{lemma}

\begin{proof}
    Let $d = \deg(v_t)$. Note that $\bm{\nu}^{(n)} + L\bm{e}_t^{(n)}$ is self-reachable on $T$ because adding chips to the self-reachable configuration $\bm{\nu}^{(n)}$ must yield another self-reachable configuration on $T$. By Lemma \ref{Max Chips Near-Minimal}, $\bm{\nu}^{(n)}$ has exactly $\ell + d + 1 - n$ chips on $v_t$, meaning $\bm{\nu}^{(n)} + L\bm{e}_t^{(n)}$ has exactly $\ell + L + d + 1 - n$ chips on $v_t$. But since $\bm{\nu}^{(n)} + L\bm{e}_t^{(n)}$ has $\ell + L$ chips in total and is self-reachable on $T$, Lemma \ref{Max Chips Near-Minimal} implies that $\bm{\nu}^{(n)} + L\bm{e}_t^{(n)}$ is a near-minimally self-reachable configuration about $v_t$ with $\ell + L$ chips, as desired.
\end{proof}

The following lemma demonstrates how to preserve the property of near-minimally self-reachable when adding a new leaf vertex to an existing tree.

\begin{lemma}\label{Add Chips Near-Minimal Leaf}
   For $n\geq 2$, let $T$ be an $n$-vertex tree  where $v_n$ is a leaf connected to $v_{n - 1}$. Suppose that $\bm{\nu}^{(n - 1)}$ is a near-minimally self-reachable configuration on $T \setminus v_n$ about $v_t$ with $\ell$ chips for some $t \in \mathbb{N}_{[1,n-1]}$ and $\ell \in \mathbb{N}_n$. Then:

    \begin{enumerate}
        \item If $t \neq n - 1$, then $\left(\bm{\nu}^{(n - 1)} + \bm{e}_{n - 1}^{(n - 1)}, 0\right)$ and $\left(\bm{\nu}^{(n - 1)}, 1\right)$ are near-minimally self-reachable configurations on $T$ about $v_t$ with $\ell + 1$ chips.

        \item If $t = n - 1$, then for all $L \in \mathbb{N}_1$, $\left(\bm{\nu}^{(n - 1)} + L\bm{e}_{n - 1}^{(n - 1)}, 0\right)$ is a minimally self-reachable configuration on $T$ about $v_{n - 1}$ with $\ell + L$ chips.
    \end{enumerate}
\end{lemma}

\begin{proof}
    To see (1), we suppose $t \neq n - 1$, and we let $d = \deg^{(T \setminus v_n)}(v_t)$. Note that both $\left(\bm{\nu}^{(n - 1)} + \bm{e}_{n - 1}^{(n - 1)}, 0\right)$ and $\left(\bm{\nu}^{(n - 1)}, 1\right)$ are self-reachable on $T$ by Corollary \ref{AddChipToLeafOrNeighbor}, and both of these configurations have $\ell + 1$ chips on $T$. Since $\bm{\nu}^{(n - 1)}$ is a near-minimally self-reachable configuration on $T \setminus v_n$ about $v_t$ with $\ell$ chips, Lemma \ref{Max Chips Near-Minimal} implies that $\bm{\nu}^{(n - 1)}$ has $\ell + d + 1 - (n - 1) = (\ell + 1) + d + 1 - n$ chips on $v_t$. Therefore, both $\left(\bm{\nu}^{(n - 1)} + \bm{e}_{n - 1}^{(n - 1)}, 0\right)$ and $\left(\bm{\nu}^{(n - 1)}, 1\right)$ have $(\ell + 1) + d + 1 - n$ chips on $v_t$. Also note that $v_t$ still has degree $d$ in $T$ because $t \neq n - 1$. So Lemma \ref{Max Chips Near-Minimal} implies that both $\left(\bm{\nu}^{(n - 1)} + \bm{e}_{n - 1}^{(n - 1)}, 0\right)$ and $\left(\bm{\nu}^{(n - 1)}, 1\right)$ are near-minimally self-reachable on $T$ about $v_t$, as desired.

       To show (2), suppose $t = n - 1$. By the same argument as in the proof of the previous claim, $\bm{\nu}^{(n - 1)}$ has $\ell + d + 1 - (n - 1) = \ell + (d + 1) + 1 - n$ chips on $v_{n - 1}$. We know that $\left(\bm{\nu}^{(n - 1)} + \bm{e}_{n - 1}^{(n - 1)}, 0\right)$ is self-reachable on $T$ by Corollary \ref{AddChipToLeafOrNeighbor}. We can add an additional $L - 1$ chips to $v_{n - 1}$ starting from this configuration and maintain a self-reachable configuration on $T$ by Proposition \ref{Add Chips to SRC} since $L \in \mathbb{N}_1$. Thus, $\left(\bm{\nu}^{(n - 1)} + L\bm{e}_{n - 1}^{(n - 1)}, 0\right)$ is a self-reachable configuration on $T$ with $(\ell + L) + (d + 1) + 1 - n$ chips on $v_{n - 1}$ and $\ell + L$ chips in total. Since $T$ is an $n$-vertex tree and $v_{n - 1}$ has degree $d + 1$ on $T$, Lemma \ref{Max Chips Near-Minimal} implies that $\left(\bm{\nu}^{(n - 1)} + L\bm{e}_{n - 1}^{(n - 1)}, 0\right)$ is a near-minimally self-reachable configuration on $T$ about $v_{n - 1}$, as desired.
\end{proof}

In a similar vein to the previous result, the next lemma demonstrates one can obtain new near-minimally self-reachable configurations when adding a leaf to a tree. 
In this case, however, we demonstrate how to create such a configuration about the new leaf added. 

\begin{lemma}\label{Move Chips to Leaf}
    Let $T$ be an $n$-vertex tree with $n\geq 2$ where $v_n$ is a leaf connected to $v_{n - 1}$. Suppose that $\bm{\nu}^{(n - 1)}$ is a near-minimally self-reachable configuration on $T \setminus v_n$ about $v_t$ with $\ell$ chips for some $t \in \mathbb{N}_{[1,n-1]}$ and $\ell \in \mathbb{N}_n$. Let $\nu_t$ be the number of chips that $\bm{\nu}^{(n - 1)}$ has on $v_t$, and let $d = \deg^{(T \setminus v_n)}(v_t)$. Then $\left(\bm{\nu}^{(n - 1)} - (\nu_t - d)\bm{e}_t^{(n - 1)}, \nu_t - d + 1\right)$ is a near-minimally self-reachable configuration on $T$ about $v_n$ with $\ell + 1$ chips.
\end{lemma}

\begin{proof}
    By Lemma \ref{Near-Minimal Properties Lemma}, we know that we can remove chips from $v_t$ on $T \setminus v_n$ starting from $\bm{\nu}^{(n - 1)}$ while maintaining a self-reachable configuration on $T \setminus v_n$ until we have $d$ chips left on $v_t$. Since we have $n - 2 - d$ chips on the trees in $(T \setminus v_n) \setminus v_t$, as seen in the proof of Lemma \ref{Near-Minimal Properties Lemma}, we will have $n - 2$ chips on $T \setminus v_n$ once we have $d$ chips left on $v_t$. This means that $\bm{\nu}^{(n - 1)} - (\nu_t - d)\bm{e}_t^{(n - 1)}$ is minimally self-reachable on $T \setminus v_{k + 1}$. Now note that since $\nu_t \ge d + 1$, $\nu_t - d + 1 \ge 2 = \deg^{(T)}(v_n) + 1$. So by Lemma \ref{Near-Minimal Properties Lemma}, $\left(\bm{\nu}^{(n - 1)} - (\nu_t - d)\bm{e}_t^{(n - 1)}, \nu_t - d + 1\right)$ is a near-minimally self-reachable configuration on $T$ about $v_n$. This configuration also has $\ell + 1$ chips because of the additional chip added to $v_n$. So the proof is complete. 
\end{proof}

The lemma to follow verifies that all near-minimally self reachable configurations must be vertices of $\CP{\ell}{T}$.

\begin{lemma}\label{Near-Minimal Configurations are Vertices}
    Let $T$ be an $n$-vertex tree, and let $\ell \in \mathbb{N}_n$. Every near-minimally self-reachable configuration on $T$ about $v_i$ with $\ell$ chips is a vertex of $\CP{\ell}{T}$.
\end{lemma}

\begin{proof}
    It suffices to prove our claim for a near-minimally self-reachable configuration about $v_n$ because the vertex labeling of $T$ is arbitrary. 
    Let $\left(\bm{\nu}^{(n - 1)}, \nu_n\right)$ be near-minimal about $v_n$, and let $d = \deg(v_n)$. 
    It suffices to show that $\left(\bm{\nu}^{(n - 1)}, \nu_n\right)$ cannot be written as a nontrivial convex combination of self-reachable configurations on $T$ with $\ell$ chips.
    We suppose that
    \begin{align}\label{eq8.1}
        \left(\bm{\nu}^{(n - 1)}, \nu_n\right) = \sum_{j = 1}^{k}\lambda_j\left(\bm{s}_j^{(n - 1)}, s_{n, j}\right),
    \end{align}
    where 
    \begin{align*}
        \sum_{j = 1}^{k}\lambda_j = 1,
    \end{align*}
    $\lambda_j \ge 0$, and $\left(\bm{s}_j^{(n - 1)}, s_{n, j}\right)$ is a self-reachable configuration on $T$ with $\ell$ chips for all $j \in \mathbb{N}_{[1,k]}$. As shown in Lemma \ref{Max Chips Near-Minimal}, we must have $\nu_n = \ell + d + 1 - n$ and $s_{n, j} \le \ell + d + 1 - n$. But since we must have
    \begin{align*}
        \sum_{j = 1}^{k}\lambda_js_{n, j} = \ell + d + 1 - n,
    \end{align*}
    we must have $s_{n, j} = \ell + d + 1 - n$ for all $j \in \mathbb{N}_{[1,k]}$. Then by the same argument as in the proof of Lemma \ref{Max Chips Near-Minimal}, $\left(\bm{s}_j^{(n - 1)}, s_{n, j}\right)$ must have minimally self-reachable configurations on each of the trees in $T \setminus v_n$ for all $j \in \mathbb{N}_{[1,k]}$. Thus, $\bm{s}_j^{(n - 1)}$ is a concatenation of minimally self-reachable configurations on the trees in $T \setminus v_n$ for all $j \in \mathbb{N}_{[1,k]}$. But Theorem \ref{cubeunimodularequivalence} implies that minimally self-reachable configurations on any tree cannot be nontrivial convex combinations of each other. Therefore, we must have $\bm{s}_j^{(n - 1)}$ be the same for all $j \in \mathbb{N}_{[1,k]}$ to avoid contradicting this fact, which trivializes the convex combination in \eqref{eq8.1} and completes the proof.
\end{proof}

The following lemma shows that any self-reachable configuration can be expressed as a convex combination of near-minimally self-reachable configurations, which is a very important result for showing Theorem \ref{thm:vertexStructure}.

\begin{lemma}\label{SRC Polytope Vertex Lemma}
    Let $T$ be an $n$-vertex tree, and let $\bm{s}^{(n)}$ be a self-reachable configuration on $T$ with $\ell$ chips for some $\ell \in \mathbb{N}_n$. Then $\bm{s}^{(n)}$ can be expressed as a convex combination of near-minimally self-reachable configurations on $T$ with $\ell$ chips.
\end{lemma}

\begin{proof}
    We proceed by induction on $n$. Our claim holds for $n = 1$ because all self-reachable configurations on a 1-vertex tree are near-minimally self-reachable. 
    For our induction hypothesis, we assume that any self-reachable configuration on a $k$-vertex tree $T$ with $\ell$ chips can be expressed as a convex combination of near-minimally self-reachable configurations on $T$ with $\ell$ chips for all $\ell \in \mathbb{N}_k$. 
    We now let $T$ be a $(k + 1)$-vertex tree and, without loss of generality, assume $v_{k + 1}$ is a leaf connected to $v_k$, and we let $\left(\bm{s}^{(k)}, s_{k + 1}\right)$ be a self-reachable configuration on $T$ with $\ell$ chips for $\ell \in \mathbb{N}_{k + 1}$. We consider the following cases:

 \textbf{Case 1:} $s_{k + 1} = 0$. Then by Corollary \ref{AddChipToLeafOrNeighbor}, we have that $\bm{s}^{(k)} - \bm{e}_k^{(k)}$ is a self-reachable configuration on $T \setminus v_{k + 1}$ with $\ell - 1$ chips. Since $\ell - 1 \in \mathbb{N}_k$, this guarantees that
        \begin{align*}
            \bm{s}^{(k)} - \bm{e}_k^{(k)} = \sum_{j = 1}^{q}\lambda_j\bm{\nu}_j^{(k)},
        \end{align*}
        where $\bm{\nu}_j^{(k)}$ is a near-minimally self-reachable configuration on $T \setminus v_{k + 1}$ with $\ell - 1$ chips for each $j \in \mathbb{N}_{[1,q]}$, and the $\lambda_j$s are nonnegative and sum to 1. So by Lemma \ref{Add Chips Near-Minimal Leaf}, $\left(\bm{\nu}_j^{(k)} + \bm{e}_k^{(k)}, 0\right)$ is near-minimally self-reachable on $T$ with $\ell$ chips for each $j \in \mathbb{N}_{[1,q]}$. Furthermore,
        \begin{align*}
            \left(\bm{s}^{(k)}, s_{k + 1}\right) = \sum_{j = 1}^{q}\lambda_j\left(\bm{\nu}_j^{(k)} + \bm{e}_k^{(k)}, 0\right),
        \end{align*}
        which proves the desired claim in this case.

\textbf{Case 2:} $s_{k + 1} = \ell + 1 - k$. By Lemma \ref{Max Chips Near-Minimal}, $\left(\bm{s}^{(k)}, s_{k + 1}\right)$ is near-minimally self-reachable about $v_{k + 1}$, which proves the desired claim in this case.

\textbf{Case 3:} $1 \le s_{k + 1} \le \ell - k$. Since $s_{k + 1} \le \ell - k$, $\bm{s}^{(k)}$ is guaranteed to have at least $k$ chips on $T \setminus v_{k + 1}$. Also, $\bm{s}^{(k)}$ is self-reachable on $T \setminus v_{k + 1}$ by Corollary \ref{subtreeconfigsaresrcs}. Therefore, we can apply our induction hypothesis to $\bm{s}^{(k)}$ and deduce that
        \begin{align*}
            \bm{s}^{(k)} = \sum_{j = 1}^{q}\lambda_j\bm{\nu}_j^{(k)},
        \end{align*}
        where $\bm{\nu}_j^{(k)}$ is a near-minimally self-reachable configuration on $T \setminus v_{k + 1}$ with $\ell - s_{k + 1}$ chips for each $j \in \mathbb{N}_{[1,q]}$. We then have that
        \begin{align}\label{eq8.2}
            \left(\bm{s}^{(k)}, s_{k + 1}\right) = \sum_{j = 1}^{q}\lambda_j\left(\bm{\nu}_j^{(k)}, s_{k + 1}\right).
        \end{align}
        The desired claim will follow if we can write each summand in \eqref{eq8.2} as a convex combination of near-minimally self-reachable configurations on $T$ with $\ell$ chips. 
        First note that if $\bm{\nu}_j^{(k)}$ is near-minimally self-reachable on $T \setminus v_{k + 1}$ about $v_t$ for $t \in \mathbb{N}_{[1,k-1]}$, then $\left(\bm{\nu}_j^{(k)}, 1\right)$ is near-minimally self-reachable on $T$ about $v_t$ by Lemma \ref{Add Chips Near-Minimal Leaf}. 
        Thus, $\left(\bm{\nu}_j^{(k)} + (s_{k + 1} - 1)\bm{e}_t^{(k)}, 1\right)$ is near-minimally self-reachable on $T$ about $v_t$ by Lemma \ref{Add Chips to Near-Minimal Configuration}, and this configuration has $\ell$ chips on $T$. 
        We also know by Lemma \ref{Move Chips to Leaf} that if $\bm{\nu}_j^{(k)}$ has $\nu_{t, j}$ chips on $v_t$, then $\left(\bm{\nu}_j^{(k)} - \left(\nu_{t, j} - d\right)\bm{e}_t^{(k)}, \nu_{t, j} - d + 1\right)$ is near-minimally self-reachable on $T$ about $v_{k + 1}$, where $d = \deg^{(T \setminus v_{k + 1})}(v_t)$ We can then add $s_{k + 1} - 1$ chips to $v_{k + 1}$ by Lemma \ref{Add Chips to Near-Minimal Configuration}, and we get that $\left(\bm{\nu}_j^{(k)} - \left(\nu_{t, j} - d\right)\bm{e}_t^{(k)}, \nu_{t, j} - d + s_{k + 1}\right)$ is a near-minimally self-reachable configuration on $T$ about $v_{k + 1}$ with $\ell$ chips. Now observe that
        \begin{equation}\label{eq8.3}
            \begin{split}
                \left(\bm{\nu}_j^{(k)}, s_{k + 1}\right) &= \frac{\nu_{t, j} - d}{s_{k + 1} - 1 + \nu_{t, j} - d}\left(\bm{\nu}_j^{(k)} + (s_{k + 1} - 1)\bm{e}_t^{(k)}, 1\right) \\
                &+ \frac{s_{k + 1} - 1}{s_{k + 1} - 1 + \nu_{t, j} - d}\left(\bm{\nu}_j^{(k)} - \left(\nu_{t, j} - d\right)\bm{e}_t^{(k)}, \nu_{t, j} - d + s_{k + 1}\right).
            \end{split}
        \end{equation}
        This is a valid convex combination because $\nu_{t, j} \ge d + 1$ by Lemma \ref{Near-Minimal Properties Lemma} and $s_{k + 1} \ge 1$ by assumption. Now suppose $\bm{\nu}_j^{(k)}$ is near-minimally self-reachable about $v_k$. We let $\nu_{k, j}$ be the number of chips $\bm{\nu}_j^{(k)}$ has on $v_k$, and we let $d = \deg^{(T \setminus v_{k + 1})}(v_k)$. By the same reasoning as before, $\left(\bm{\nu}_j^{(k)} - \left(\nu_{k, j} - d\right)\bm{e}_k^{(k)}, \nu_{k, j} - d + s_{k + 1}\right)$ will be near-minimally self-reachable on $T$ about $v_{k + 1}$ and will have $\ell$ chips. Also, by Lemma \ref{Add Chips Near-Minimal Leaf}, $\left(\bm{\nu}_j^{(k)} + s_{k + 1}\bm{e}_k^{(k)}, 0\right)$ is near-minimally self-reachable on $T$ about $v_{k}$ and also has $\ell$ chips. Now observe that
        \begin{equation}\label{eq8.4}
            \begin{split}
                \left(\bm{\nu}_j^{(k)}, s_{k + 1}\right) &= \frac{\nu_{k, j} - d}{s_{k + 1} + \nu_{k, j} - d}\left(\bm{\nu}_j^{(k)} + s_{k + 1}\bm{e}_k^{(k)}, 0\right) \\
                &+ \frac{s_{k + 1}}{s_{k + 1} + \nu_{k, j} - d}\left(\bm{\nu}_j^{(k)} - \left(\nu_{k, j} - d\right)\bm{e}_k^{(k)}, \nu_{k, j} - d + s_{k + 1}\right).
            \end{split}
        \end{equation}
        Once again, this is a valid convex combination because $\nu_{t, j} \ge d + 1$ by Lemma \ref{Near-Minimal Properties Lemma} and $s_{k + 1} \ge 1$ by assumption. Taken together, \eqref{eq8.3} and \eqref{eq8.4} imply that every summand in \eqref{eq8.2} can be written as a convex combination of near-minimally self-reachable configurations on $T$ with $\ell$ chips. When we substitute these convex combinations into \eqref{eq8.2}, we obtain a convex combination of near-minimally self-reachable configurations on $T$ with $\ell$ chips that yields $\left(\bm{s}^{(k)}, s_{k + 1}\right)$.

    Since $\left(\bm{s}^{(k)}, s_{k + 1}\right)$ was able to be expressed as a convex combination of near-minimally self-reachable configurations on $T$ with $\ell$ chips in all cases, the proof is complete.
\end{proof}

We now have the necessary machinery to prove Theorem \ref{thm:vertexStructure}.

\begin{proof}[Proof of Theorem \ref{thm:vertexStructure}]
    By Lemma \ref{Near-Minimal Configurations are Vertices}, we have that near-minimally self-reachable configurations on $T$ with $\ell$ chips are vertices of $\CP{\ell}{T}$, so we only need to show that no other vertices exist. It suffices to show that any point in $\CP{\ell}{T}$ can be expressed as a convex combination of near-minimally self-reachable configurations on $T$ with $\ell$ chips. We let $\bm{x}^{(n)} \in \CP{\ell}{T}$ be arbitrary. Then
    \begin{align}\label{eq8.5}
        \bm{x}^{(n)} = \sum_{j = 1}^{k}\lambda_j\bm{s}_j^{(n)},
    \end{align}
    where $\bm{s}_j^{(n)}$ is a self-reachable configuration on $T$ with $\ell$ chips for each $j \in \mathbb{N}_{[1,k]}$ and the $\lambda_j$s are nonnegative and sum to 1. By Lemma \ref{SRC Polytope Vertex Lemma}, we can express each $\bm{s}_j^{(n)}$ as a convex combination of the near-minimally self-reachable configurations on $T$ with $\ell$ chips. We can then substitute these convex combinations into \eqref{eq8.5} to express $\bm{x}^{(n)}$ as a convex combination of near-minimally self-reachable configurations on $T$ with $\ell$ chips. 
\end{proof}

We conclude this section with an example that demonstrates some of the methods used in the proof of Lemma \ref{SRC Polytope Vertex Lemma}.

\begin{example}
    Consider the self-reachable configuration $(1, 2, 3)$ on the path graph $P_3$ (in which we label the vertices such that $v_1$ and $v_3$ are the leaves). Note that on the graph $P_2$, we have that $(3, 0)$ and $(0, 3)$ are near-minimally self-reachable, and
    \begin{align*}
        (1,2)
        = \frac{1}{3}
       (3,0)
        + \frac{2}{3}
       (0,3)
    \end{align*}
    Therefore, we have that
    \begin{align}\label{eq4.6}
        (1,2,3)
        = \frac{1}{3}
        (3,0,3)
        + \frac{2}{3}
       (0,3,3)
    \end{align}
    We now need to write $(3, 0, 3)$ and $(0, 3, 3)$ as convex combinations of near-minimally self-reachable configurations on $P_3$. Since $(3, 0)$ is near-minimally self-reachable on $P_2$ about $v_1$, we have that $(3, 0, 1)$ is near-minimally self-reachable on $P_3$ about $v_1$ by Lemma \ref{Add Chips Near-Minimal Leaf}. Therefore, $(5, 0, 1)$ is also near-minimally self-reachable on $P_3$ about $v_1$ by Lemma \ref{Add Chips to Near-Minimal Configuration}. By Lemma \ref{Move Chips to Leaf}, we also have that $(1, 0, 5)$ is near-minimally self-reachable on $P_3$ about $v_3$. The convex combination generated by \eqref{eq8.3} is
    \begin{align}\label{eq4.7}
        (3,0,3)
        =
        \frac{1}{2}
        (5,0,1)
        + \frac{1}{2}
        (1,0,5)
    \end{align}
    Now note that since $(0, 3)$ is near-minimally self-reachable on $P_2$ about $v_2$, we have that $(0, 4, 0)$ is near-minimally self-reachable on $P_3$ about $v_2$ by Lemma \ref{Add Chips Near-Minimal Leaf}. So by Lemma \ref{Add Chips to Near-Minimal Configuration}, $(0, 6, 0)$ is also near-minimally self-reachable on $P_3$ about $v_2$. By Lemma \ref{Move Chips to Leaf}, we also have that $(0, 1, 5)$ is near-minimally self-reachable on $P_3$ about $v_3$. The convex combination generated by \eqref{eq8.4} is
    \begin{align}\label{eq4.8}
       (0,3,3)
        =
        \frac{2}{5}
        (0,6,0)
        + \frac{3}{5}
        (0,1,5)
    \end{align}
    When we substitute \eqref{eq4.7} and \eqref{eq4.8} into \eqref{eq4.6}, we obtain
    \begin{align}\label{eq4.9}
        (1,2,3)
        = \frac{1}{6}
       (5,0,1)
        + \frac{1}{6}
        (1,0,5)
        + \frac{4}{15}
        (0,6,0)
        + \frac{2}{5}
        (0,1,5)
    \end{align}
Since \eqref{eq4.9} expresses $(1, 2, 3)$ as a convex combination of near-minimally self-reachable configurations on $P_3$ with 6 chips, the example is complete.
\end{example}

\end{document}